\documentclass{elsarticle}

\usepackage{amsthm,amsmath,amssymb,url,color,array}

\newtheorem{thm}{Theorem}[section]
\newtheorem{lem}[thm]{Lemma}

\newtheorem{prop}[thm]{Proposition}

\theoremstyle{definition}

\newtheorem{pr}[thm]{Problem}

\theoremstyle{remark}

\numberwithin{equation}{section}

\newcommand{\set}[1]{\big\{#1\big\}}
\def\st{:\,}
\newcommand{\id}{\gamma^\text{\tiny ID}}
\newcommand{\loc}{\gamma^\text{\tiny LOC}}
\newcommand{\cg}[2]{C_{#1}(#2)}

\newcommand{\ceil}[1]{\left\lceil #1\right\rceil}
\renewcommand{\d}{{\mathrm d}}
\def\sh{\gamma}
\def\iso{\cong}
\def\Z{\mathbb{Z}}
\def\le{\leqslant}
\def\ge{\geqslant}

\begin{document}

\begin{frontmatter}

\title{Locating and Identifying Codes in Circulant Networks}

\author[mg]{M. Ghebleh}
\address[mg]{Department of Mathematics,
             Faculty of Science,
             Kuwait University,
             State of Kuwait}
\ead[mg]{mamad@sci.kuniv.edu.kw}

\author[ln]{L. Niepel}
\address[ln]{Department of Computer Science,
             Faculty of Science,
             Kuwait University,
             State of Kuwait}
\ead[ln]{niepel@sci.kuniv.edu.kw}

\begin{abstract}
A set $S$ of vertices of a graph $G$ is a dominating set of $G$ if
every vertex $u$ of $G$ is either in $S$ or it has a neighbour in~$S$.
In other words $S$ is dominating if the sets $S\cap N[u]$ where $u\in V(G)$
and $N[u]$ denotes the closed neighbourhood of $u$ in $G$, are all nonempty.
A set $S\subseteq V(G)$ is called a {\em locating code} in $G$, if the sets
$S\cap N[u]$ where $u\in V(G)\setminus S$ are all nonempty and distinct.
A set $S\subseteq V(G)$ is called an {\em identifying code} in $G$,
if the sets $S\cap N[u]$ where $u\in V(G)$ are all nonempty and distinct.
We study locating and identifying codes in the circulant networks~$\cg{n}{1,3}$.
For an integer $n\ge7$, the graph $\cg{n}{1,3}$ has vertex set $\Z_n$
and edges $x y$ where $x,y\in\Z_n$  and $|x-y|\in\{1,3\}$.
We prove that a smallest locating code in $\cg{n}{1,3}$ has size $\lceil n/3\rceil+c$,
where $c\in\{0,1\}$, and a smallest identifying code in $\cg{n}{1,3}$ has size
$\lceil4n/11\rceil+c'$, where $c'\in\{0,1\}$.
\end{abstract}

\end{frontmatter}

\section{Introduction}
\label{sec:intro}

All graphs considered in this paper are simple,without multiple
edges or loops. Given a graph $G=(V,E)$, for any vertex $u\in V$,
{we denote the {\em neighbourhood} of $u$ in~$G$ by
$N_G(u)=\set{x\in V\st ux\in E}$. By the {\em closed neighbourhood}
of $u\in V$, we mean the set $N_G[u]=N_G(u)\cup\{u\}$.
When the graph $G$ is clear from the context, we omit the subscripts
in this notation.}
Given a subset $S\subseteq V$, the {\em shadow} of a
vertex $u\in V$ on $S$ is defined to be the set $S_u=N[u]\cap S$.
The set $S$ is a {\em dominating set} of $G$ if every $u\in V$ has
a nonempty shadow on~$S$. The set $S$ is said to be an {\em
identifying code}, if it is dominating, and distinct vertices
$u,v\in V$ have distinct shadows on~$S$. The smallest size of an
identifying code in a graph $G$ (if one exists) is called the {\em
identifying number} of $G$ and is denoted by~$\id(G)$. The set $S$
is said to be a {\em locating-dominating set} or a {\em locating
code}, if it is dominating, and distinct vertices $u,v\in
V\setminus S$ have distinct shadows on~$S$. The smallest size of a
locating code in a graph $G$ is called the {\em locating number}
of $G$ and is denoted by~$\loc(G)$. Locating codes were first
introduced in~\cite{slater88}, motivated by nuclear power plant
safety. Vertices of a locating-dominating set $S$ correspond to
safeguards that are able to locate an intruder corresponding to a
vertex in $V-S$. Identifying codes were first introduced in more
general form in~\cite{kcl98}. Karpovsky~{\it et al.} study $r$--identifying
codes in specific topologies of interest in distributed computing
for diagnosis of faulty units in multi-processor networks. In the
definition of $r$--identifying codes and $r$--locating-dominating sets
the neighbourhood $N[u]$ is replaced by the set
$N_r[u]=\{x \in V\st d(u,x) \le r\}$ for a constant $r \ge 1$,
where $d(u,x)$ is the graph distance between vertices $u$ and $x$.
The $r$--identifying and
$r$--locating codes correspond to the identifying and locating codes
in the {$r$th power $G^r$ of $G$}. Locating and identifying codes
have received a great deal of attention from
researchers~\cite{rs84,carson95,slater95,fh98,chlz99,bhl01,slater02}.
In particular, locating and identifying codes in special classes
of networks have been studied. Examples of such articles include
locating codes in trees~\cite{howard04,hhh06}, locating codes in
infinite grids~\cite{hl06}, locating codes in series-parallel
networks~\cite{css86}, locating codes in the infinite triangular
grid~\cite{honkala06}, identifying codes in the infinite hexagonal
grid~\cite{cranstonyu09}, identifying codes in
cages~\cite{laihonen08}, identifying codes in binary Hamming
spaces~\cite{ejlr10}, identifying and locating codes in geometric
networks~\cite{mullersereni09}.

Given positive integers $n$ and $d_1,\ldots,d_k<n/2$,
we define the {\em circulant graph} $\cg{n}{d_1,\ldots,d_k}$
to have vertex set $\Z_n=\big\{0,1,\ldots,n-1\big\}$, in which two
vertices $x,y$ are adjacent if and only if $|x-y|\in\set{d_1,\ldots,d_k}$.
For positive integers $d_1,\ldots,d_k$, the infinite circulant graph
$\cg{\infty}{d_1,\ldots,d_k}$ is defined on the vertex set $\Z$
with edges $x y$ such that $|x-y|\in\set{d_1,\ldots,d_k}$.
The {\em density} of $S\subseteq\Z$ in $\Z$ is defined by
\[\rho(S)=\limsup_{N}\frac{\big|S\cap[-N,N]\big|}{2N+1}.\]
Identifying and locating codes of the circulant graphs
$\cg{n}{1,2,\ldots,r}$ are studied
in~\cite{ber04,gra06,xu08,rob08,clm09,exoo11} as $r$--locating and
$r$--identifying codes of cycles.
{The values of $\loc(\cg{n}{1,2})$ are established in~\cite{clm09}:
for $n \ge 6$,}
$$\lceil n/3 \rceil \le \loc(\cg{n}{1,2})\le \lceil n/3 \rceil \ + 1.$$
{The values of $\id(\cg{n}{1,2})$ are established in~\cite{rob08}:
for $n\ge 8$,}
$$\lceil n/2 \rceil \le \id(\cg{n}{1,2})\le \lceil n/2 \rceil +2.$$
 Motivated by these results, we study locating and
identifying codes of the circulant graphs $\cg{n}{1,3}$. We prove
{for $n\ge9$,
$$\lceil n/3 \rceil\le\loc\big(\cg{n}{1,3}\big)\le\lceil{n}/{3}\rceil+1,$$
and
$$\lceil 4n/11 \rceil\le\id\big(\cg{n}{1,3}\big)\le\lceil 4n/11 \rceil+1.$$
We also prove that the least density of a locating (resp. identifying) code
in $\cg{\infty}{1,3}$ is $1/3$ (resp. $4/11$).
}

\section{General lower bounds}

Recall that for a graph $G=(V,E)$ and a dominating set $S\subset V$, by the
shadow of a a vertex $u\in V$ on $S$ we mean the set $S_u=S\cap N[u]$.
The {\em profile} of $u\in V$ to be the $\d_G(u)+1$--tuple $\pi(u)$
with entries $|S_x|$ where $x\in N[u]$, in ascending order.
The {\em share} of a vertex $u\in S$ in $S$ is defined by
\[\sh(u;S)=\sum_{x\in N[u]}\frac{1}{|S_x|}.\]
When the set $S$ is clear from the context, we refer to $\sh(u;S)$
simply as the share of~$u$ and we denote it by $\sh(u)$. The
following lemma, proved by a simple double-counting argument, is a
powerful tool in obtaining lower bounds on (various flavors of)
domination numbers.

\begin{lem}{\rm\cite{slater02}}
Let $G$ be a graph of order $n$ and let $S$ be a dominating set of~$G$. Then
$\displaystyle\sum_{u\in S}\sh(u)=n$.
\label{lem:slater}
\end{lem}

The above lemma yields the following lower bounds on the size of
locating and identifying codes in a general graph.

\begin{prop}{\rm\cite{slater88}}
For a graph $G$ of order $n$ and maximum degree $\Delta$ we have
 $\loc(G)\ge\displaystyle\frac{2n}{\Delta+3}$. \label{prop:genboundl}
\end{prop}

\begin{prop}{\rm\cite{kcl98}}
For a graph $G$ of order $n$ and maximum degree $\Delta$ we have
$\id(G)\ge \displaystyle\frac{2n}{\Delta+2}$.
\label{prop:genboundi}
\end{prop}

\section{Locating number of $\cg{n}{1,3}$}

From Proposition~\ref{prop:genboundl} it follows that if $G$ is a
$4$--regular graph of order $n$ then  $\loc(G)\ge 2n/7 $. In this
section we obtain a better lower bound for the locating number of
the circulant graphs $\cg{n}{1,3}$, and we show that this bound is
asymptotically tight. Let $n\ge13$, and let $S$ be a locating code
in the graph $\cg{n}{1,3}$. A vertex $u\in S$ is said to be {\em
heavy} if $\sh(u)>3$.

\begin{lem}
Let $u\in S$ be a heavy vertex. Then $\pi(u)$ is either
$(1,1,2,2,3)$ or $(1,1,2,3,4)$. Moreover, we may assign to each
heavy vertex $u\in S$, a vertex $u'\in S$, called the {\em
mate} of $u$, such that $\sh(u)+\sh(u')\le6$.
{Moreover, distinct heavy vertices have distinct mates.}
\label{lem:locproof}
\end{lem}

\begin{proof}
By symmetry, we may assume $u=0$. Note that if there is at most
one $x\in N[0]$ with $|S_x|=1$, then {$\sh(0)\le 1+4/2=3$}.
Thus there is $x\in N(0)$ such that $S_0=S_x=\{0\}$.
Since $N(0)=\{-3,1,1,3\}$, we may assume without loss of generality that
$x\in\{-1,-3\}$.

{\em Case 1:} $x=-1$. Since $S_0=S_{-1}=\{0\}$, we have $[-4,3]\cap S=\{0\}$.
Thus we must have $-6\in S$ since otherwise, $S_{-3}=\{0\}$,
which contradicts the locating property of~$S$.
Similarly, we must have $4\in S$ since otherwise, $S_{1}=\{0\}$.
We must also have $6\in S$ since otherwise, $S_1=S_3=\{0,4\}$.
We now have $|S_0|=|S_{-1}|=1$, $|S_{-3}|=|S_{1}|=2$, and $|S_3|=3$,
giving $\pi(0)=(1,1,2,2,3)$.
Moreover, we must have $5\in S$ since otherwise, $S_2=\emptyset$.

{\em Case 2:} $x=-3$. Since $S_0=S_{-3}=\{0\}$, we have
$\{-6,-4,-3,-2,-1,1,3\}\cap S=\emptyset$.
Thus we must have $2\in S$ since otherwise, $S_{-1}=\{0\}$.
Since $2$ is a common neighbour of $-1$, $1$, and $3$,
we must have $4\in S$, and since $4$ is a common neighbour of
$1$ and $3$, we must have $6\in S$.
We now have $|S_0|=|S_{-3}|=1$, $|S_{-1}|=2$, $|S_{1}|=3$, and $|S_3|=4$,
giving $\pi(0)=(1,1,2,3,4)$.

In case~1, we assign $4$ as the mate of~$0$.
We have $|S_1|,|S_4|,|S_7|\ge2$ and $|S_3|,|S_5|\ge3$, thus
$\sh(4)\le13/6$. We see that $\sh(0)+\sh(4)\le11/2<6$.
In case~2, we assign $2$ as the mate of~$0$.
We have $|S_2|\ge1$, $|S_{-1}|=2$, $|S_1|,|S_5|\ge3$, and $|S_3|=4$,
thus $\sh(2)\le29/12$. We see that $\sh(0)+\sh(2)\le11/2<6$.

Note that in case~1, the vertices $1,2,3$ between $0$
and its mate $4$ are not in~$S$,
and in case~$2$, the vertex $1$ between $0$ and its mate $2$
is not in~$S$. Since in case~1, $5\in S$ and in case~$2$, $4\in S$,
we see that in either case, $u'$ cannot also be the mate of~$u'+4$.
Also since $5\in S$ in case~1, we see that in this case, $u'$ cannot
also be the mate of~$u'+2$.
It remains to show that in case~2, $u'=2$ is not also the
mate of $4$. This is true since all neighbours of $4$
have a shadow of size at least $2$, hence $4$ is not a heavy vertex.
\end{proof}

\begin{thm}
For every $n\ge13$, $\loc\big(\cg{n}{1,3}\big)\ge n/3$.
\label{thm:locmain}
\end{thm}

\begin{proof}
Let $S$ be a locating code in $\cg{n}{1,3}$.
Lemma~\ref{lem:locproof} gives a unique mate $u'$ for every heavy vertex~$u$,
such that $\sh(u)+\sh(u')\le6$. On the other hand, for every other vertex
$v\in S$ we have $\sh(v)\le3$. Thus the total share of vertices of $S$ is
at most $3|S|$.
The result now follows from Lemma~\ref{lem:slater}.
\end{proof}

Note that the proof of Lemma~\ref{lem:locproof} works also for the
graph $\cg{\infty}{1,3}$. On the other hand, the neighbours of each vertex
$u\in\Z$ are within short numeric distances of $u$ (at most $3$).
These allow us to prove a lower bound of $1/3$ on the density of
any locating set in $\cg{\infty}{1,3}$.

\begin{thm}
Every locating code in $\cg{\infty}{1,3}$ has density at least ${1}/{3}$.
\label{thm:infloc}
\end{thm}

\begin{proof}
Let $S$ be a locating set in $\cg{\infty}{1,3}$.
Note that the mate of each heavy vertex found in Lemma~\ref{lem:locproof}
is within numeric distance at most $4$ of that vertex.
Thus for any positive integer $N$, the set $S'=S\cap[-N,N]$ contains at most
two heavy vertices (one at each end) whose mate is not present in~$S'$.
Since by Lemma~\ref{lem:locproof}, the share in $S$ of a heavy vertex
is at most $3+1/3$, we obtain
\[\sum_{u\in S'}\sh(u)\le3|S'|+2/3.\]
On the other hand,
\[\sum_{u\in S'}\sh(u)=\sum_{u\in S'}\sum_{x\in N[u]}\frac{1}{|S_x|}
  \ge\sum_{x\in[-N,N]}\sum_{u\in S_x}\frac{1}{|S_x|}=2N+1.\]
The inequality appears since not all neighbours of every $u\in S'$
are necessarily in the range $[-N,N]$. These inequalities give
$2N+1\le 3|S'|+2/3$, or
\[\frac{\big|S\cap[-N,N]\big|}{2N+1}\ge\frac{1}{3}-\frac{2}{9(2N+1)}.\]
We conclude that $\rho(S)\ge1/3$.
\end{proof}

In the remainder of this section, we provide constructions of
locating codes in circulant graphs $\cg{n}{1,3}$. From
Theorem~\ref{thm:locmain}, we know that such codes have size at
least $\ceil{n/3}$. We give general constructions for $n\ge13$.
These codes have size $\ceil{n/3}$, unless when $n\equiv2\mod3$,
where the constructed code has size $\ceil{n/3}+1$. We do not know
whether this is best possible, but using a brute-force computer
search, we verified that for $14\le n\le38$, a locating code of
size $\ceil{n/3}$ does not exist in this case. For $n<13$, we
verified using this program that $\loc\big(\cg{7}{1,3}\big)=3$,
$\loc\big(\cg{8}{1,3}\big)=6$,
$\loc=\big(\cg{9}{1,3}\big)=\loc\big(\cg{10}{1,3}\big)=\loc\big(\cg{11}{1,3}\big)=4$,
and $\loc\big(\cg{12}{1,3}\big)=5$.

For a positive integer $t$, let
\[A_t=\big\{6i+j\st 0\le i\le t-1\text{ and }j\in\{0,1\}\big\}.\]
It is easy to see that for $t\ge3$, the set $A_t$ is a locating code in $\cg{6t}{1,3}$.
The sets $A_t$ can indeed be used in constructions of locating codes for
the graphs $\cg{n}{1,3}$ when $n$ is not necessarily a multiple of~$6$.
Such constructions are presented in Table~\ref{tbl:loc}.
\begin{table}[ht]
\centering
\begin{tabular}{lcl}
\hline
$n$ &&  A locating code for $\cg{n}{1,3}$\\
\hline
$6t+1$  && $A_t\cup\{6t-2\}$\\
$6t+2$  && $A_{t+1}$\\
$6t+3$  && $A_{t+1}$\\
$6t+4$  && $A_{t+1}$\\
$6t+5$  && $A_{t+1}\cup\{6t-2\}$\\
$6t+6$  && $A_{t+1}$\\
\hline
\end{tabular}
\caption{Constructions of locating codes for the circulant graphs $\cg{n}{1,3}$.
         Here $t\ge 2$ is an integer.}
\label{tbl:loc}
\end{table}

We omit the proofs here. The proofs are straight-forward, and all
take advantage of the ``local'' structure of the graphs
$\cg{n}{1,3}$, namely the fact that each neighbourhood is
contained in an interval of length~$6$. We present an example of
these codes in Figure~\ref{fig:loceg}. These results are
summarized in the next theorem.

\begin{figure}[ht]
\centering
\includegraphics{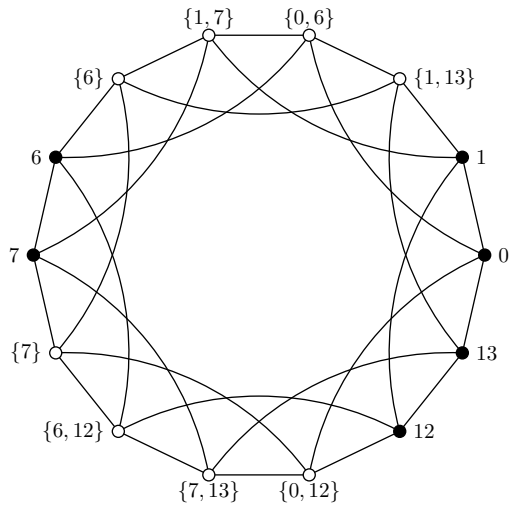}
\caption{A minimum locating code of $\cg{14}{1,3}$.
        Vertices in the code are in black.
        The number next to a vertex is its label.
        The set next to each vertex is its shadow on this code.}
\label{fig:loceg}
\end{figure}

\begin{thm}
Let {$n\ge9$}. Then $\loc\big(\cg{n}{1,3}\big)=\ceil{n/3}$ if $n\not\equiv2\mod3$,
and ${\ceil{n/3}\le}\loc\big(\cg{n}{1,3}\big)\le\ceil{n/3}+1$ if $n\equiv2\mod3$.
\label{thm:locconst}
\end{thm}

For a locating code in $\cg{\infty}{1,3}$ with density $1/3$, one may take
the code \[A_\infty=\big\{6i+j\st i\in\Z\text{ and }j\in\{0,1\}\big\}.\]

\section{Identifying number of $\cg{n}{1,3}$}

\def\H{\Gamma}

In this section we obtain a lower bound for the identifying number of
the circulant graphs $\cg{n}{1,3}$, and we show that this bound is
asymptotically tight.
We assume that $n\ge13$ is an integer,
and $S$ is an identifying code in the circulant graph $G=\cg{n}{1,3}$.
A vertex $u\in S$ is said to be a {\em heavy} vertex, if $\sh(u)>11/4$ .
The subgraph of $G$ induced by $S$ is denoted by~$\H$.
The connected component of $\H$ containing a vertex $u\in S$ is denoted by~$\H_u$.
By a {\em heavy component} of $\H$, we mean a connected
component whose vertices have average share larger than~$11/4$.

\begin{lem}
Let $u\in S$ be a heavy vertex.
Then $\pi_u=(1,2,2,2,3)$.
\end{lem}

\begin{proof}
If $\pi(u)$ contains at least two numbers greater than $2$, then
$$\sh(v)\le1+\frac{1}{2}+\frac{1}{2}+\frac{1}{3}+\frac{1}{3}<\frac{11}{4}.$$
This contradicts the choice of $u$ as a heavy vertex.
Then $\pi(u)=(1,2,2,2,a)$ for some integer~$a\ge2$.
Moreover, if $a\ge4$, then $\sh(u)\le11/4$.
Hence $\pi(u)$ is either $(1,2,2,2,2)$ or $(1,2,2,2,3)$.

Suppose that $\pi(u)=(1,2,2,2,2)$. We may assume $u=0$.
By the assumption on $\pi(0)$, the shadow $S_0$ has size $1$ or~$2$.
If $|S_0|=1$, then $N[0]\cap S=\{0\}$, and $|S_x|=2$
for all $x\in N(v)$.
In particular, $S_1=\{0,y\}$ where $y\in\{-2,2,4\}$. If $y=\pm2$, then
$S_{-1}=S_1$, and if $y=4$, then $S_{1}=S_{3}$.
These both contradict the identifying property of~$S$.
If $|S_0|=2$, let $S_0=\{0,x\}$. Then $\{0,x\}\subseteq S_x$, thus $|S_x|\ge2$.
Since $|S_x|\le2$, this gives $S_x=S_0$, a contradiction.
\end{proof}

Note that if $u\in S$, then $\d_\H(u)=|S_u|-1$. On the other hand,
if $u$ is a heavy vertex, it has profile $(1,2,2,2,3)$ by the above lemma,
so $|S_u|\le3$. We conclude that $\d_\H(u)\in\{0,1,2\}$.
We first prove that $\d_\H(u)\not=0$ when $u$ is a heavy vertex.

\begin{lem}
Let $u\in S$ be a heavy vertex. Then $u$ is not an isolated vertex in~$\H$.
\end{lem}

\begin{proof}
We may assume $u=0$.
Suppose $\d_\H(0)=0$, namely, $\{-3,-1,1,3\}\cap S=\emptyset$.
Since $\pi(0)=(1,2,2,2,3)$, we may assume that $|S_1|=|S_3|=2$.
Now $2\not\in S$ and $4\not\in S$, since otherwise,
we have $S_1=S_3$, a contradiction.
Hence $-2,6\in S$. This is a contradiction since $\{-2,0\}$ is now contained in
$S_{-3}$, $S_{-1}$, and $S_1$, while these sets are distinct and two of them have
size~$2$.
\end{proof}

\begin{lem}
If $u\in S$ is a heavy vertex with $\d_\H(u)=2$, then $\H_u$ is isomorphic
to $P_2$, the path graph of length~$2$.
\label{lem:bigcomponent}
\end{lem}

\begin{proof}
Let $N_\H(u)=\{v,w\}$. Then $\d_\H(v)\ge1$ and $\d_\H(w)\ge1$. Since $|S_u|=3$ and
$u$ is heavy, we must have $\d_\H(v)=1$ and $\d_\H(w)=1$. Thus none of the vertices
in the set $\{u,v,w\}$ has a neighbour outside this set in~$\H$.
\end{proof}

\begin{lem}
Every heavy component of $\H$ is isomorphic to $P_2$.
\label{lem:P2}
\end{lem}

\begin{proof}
Suppose that $\H$ has a heavy component with at least $4$ vertices.
Then this component has at least one heavy vertex~$u$.
By Lemma~\ref{lem:bigcomponent}, we have $\d_\H(u)=1$. Let $N_\H(u)=\{v\}$.
The vertex $v$ is called the {\em mate} of~$u$.
Since $u$ is heavy and since $\H_u$ has order at least $4$, we have $\d_\H(v)=2$.
Let $N_\H(v)=\{u,w\}$. Since $\H_v$ has order at least $4$, we have $\d_\H(w)\ge2$,
thus $|S_w|\ge3$. This shows that
$\sh(v)\le\frac{1}{1}+\frac{2}{2}+\frac{2}{3}=\frac{8}{3}$,
which in turn gives $\sh(u)+\sh(v)\le11/2$.
On the other hand, since $|S_v|=3$ and $|S_w|\ge3$, $w$ is not heavy.
Therefore, $v$ is not the mate of two heavy vertices.
Averaging $\sh(x)$ over all $x\in V(\H_u)$, we see that each heavy vertex and
its mate contribute $11/2$ together.
Since every other vertex has share less than ${11}/{4}$,
the average share of the vertices of $\H_u$ is at most~${11}/{4}$.
This contradicts the assumption that $\H_u$ is heavy.
\end{proof}

\begin{lem}
Every heavy component of $\H$ is isomorphic to a path of length $2$,
all whose vertices are heavy.
Moreover, the vertices of this component are of the form $\{u-1,u,u+3\}$ 
{or $\{u-3,u,u+1\}$} for some {$u\in\Z_n$}.
\label{lem:P2more}
\end{lem}

\begin{proof}
Consider a heavy component of $\H$ and let $W$ denote its vertex
set. Without loss of generality, we may assume that $0\in W$, thus
we may refer to this component as~$\H_0$. By Lemma~\ref{lem:P2},
we know that $|W|=3$. By symmetries of $\cg{\infty}{1,3}$, we may
assume that $W$ is one of the sets $\{-3,0,3\}$, $\{-1,0,1\}$,
$\{-1,0,2\}$, or $\{-1,0,3\}$. In what follows, we show that the
first three of these choices yield to a contradiction.

If $W=\{-3,0,3\}$, then $\{\pm1,\pm2,\pm4,\pm6\}\cap S=\emptyset$.
This gives $S_{-1}=S_{1}=\{0\}$, a contradiction.

If $W=\{-1,0,1\}$, then $\{\pm2,\pm3,\pm4\}\cap S=\emptyset$.
If none of the vertices $-5,5$ is in $S$, we have $S_{-2}=S_2=\{-1,1\}$,
a contradiction. Therefore, at least one of $-5$ and $5$, say $5$, is in $S$.
This gives $|S_{-2}|,|S_4|\ge2$, $|S_0|=|S_2|=3$, and $|S_1|=2$,
hence $\sh(1)\le{13}/{6}$. We now obtain
$\frac{1}{3}\big(\sh(-1)+\sh(0)+\sh(1)\big)\le{47}/{18}<{11}/{4}$.
This contradicts the choice of $\H_0$.

If $W=\{-1,0,2\}$, then $6\in S$ since otherwise, $S_1=S_3=\{0,2\}$.
This gives $\sh(2)\le13/6$, which yields a contradiction similarly to the previous
case.

Therefore, if $\H_0$ is heavy, then $W=\{-1,0,3\}$, up to symmetries of $\cg{n}{1,3}$.
It remains to prove that all vertices in $W$ are heavy.
With $W=\{-1,0,3\}$ we obtain $\{-4,-3,-2,1,2,4,6\}\cap S=\emptyset$.
Moreover, $-5\not\in S$ since otherwise, $|S_{-4}||S_{-2}|,|S_{-1}|,|S_2|\le2$ and
$|S_0|=3$ which give $\sh(-1)\le{7}/{3}$. This is a contradiction with the choice of~$\H_0$.
Similarly,
$5\not\in S$ and $7\not\in S$ since otherwise, $\sh(3)\le{7}/{3}$.
On the other hand, $-6\in S$ since otherwise, $S_{-3}=S_1=\{0\}$,
$-7\in S$ since otherwise, $S_{-4}=S_{-2}=\{-1\}$,
and $9\in S$ since otherwise, $S_4=S_6=\{3\}$.
We obtain $\pi(-1)=\pi(0)=\pi(3)=(1,2,2,2,3)$.
\end{proof}

\begin{thm}
For every $n\ge13$ we have $\id\big(\cg{n}{1,3}\big)\ge 4n/11$.
\label{thm:idmain}
\end{thm}

\begin{proof}
Let $S$ be an identifying code in $\cg{n}{1,3}$, where $n\ge13$.
We assign to each heavy component of $\H$, a unique subset of $S$
referred to as the {\em mate} of that component, such that the average
share of vertices in a heavy component and its mate together is at most~$11/4$.

Let $\H_u$ be a heavy component of $\H$. By the proof of Lemma~\ref{lem:P2more},
and by symmetry, we may assume that $u=0$ and $V(\H_0)=W=\{-1,0,3\}$.
Then $S\cap[-7,10]=\{-7,-6,-5,-1,0,3,8,9,10\}$. This is by the proof of
Lemma~\ref{lem:P2more}, and that if $8\not\in S$,
then $S_5=\emptyset$. Also if $10\not\in S$, then $S_5=S_7=\{8\}$.

If $\{11,12,13\}\cap S\not=\emptyset$, the mate of $\H_0$ is defined to be
the set $W'=\{8,9\}$.
If $11\in S$, then $\sh(8)\le5/2$ and $\sh(9)\le2$.
If $12\in S$, then $\sh(8)\le31/12$ and $\sh(9)\le9/4$.
If $13\in S$, then $\sh(8)\le17/6$ and $\sh(9)\le13/6$.
In either of these cases we have
$$\frac{1}{5}\Big(\sh(-1)+\sh(0)+\sh(3)+\sh(8)+\sh(9)\Big)\le\frac{11}{4}.$$

If $\{11,12,13\}\cap S=\emptyset$, we see that $\H_8\iso P_2$ and by the proof of
Lemma~\ref{lem:P2more}, we have $\frac{1}{3}(\sh(8)+\sh(9)+\sh(10))\le\frac{47}{18}$.
In this case we assign $W''=\{8,9,10\}$ as the mate of $\H_0$, and we have
$$\frac{1}{6}\Big(\sh(-1)+\sh(0)+\sh(3)+\sh(8)+\sh(9)+\sh(10)\Big)
 \le\frac{1}{2}\Big(\frac{17}{6}+\frac{47}{18}\Big)<\frac{11}{4}.$$

The mates defined above do not contain any mates assigned in the proof of Lemma~\ref{lem:P2},
since those are adjacent to a heavy vertex.
On the other hand, since there are four vertices $4,5,6,7$ not in $S$,
between $W$ and each of $W'$ and $W''$ we see that $W'$ does not overlap with any
mate assigned to other heavy components.
Moreover, if $14\not\in S$, then $S_{7}=S_{11}=\{8,10\}$, a contradiction.
Thus $W''$ does not overlap any other mates (the four vertex gap is not present after $W''$).

We conclude that the average share of the vertices of $S$ is at most $11/4$,
which by Lemma~\ref{lem:slater} gives $|S|\ge 4n/11$.
\end{proof}

Similarly to the proof of Theorem~\ref{thm:infloc}, we may prove the following theorem.

\begin{thm}
Every identifying code in $\cg{\infty}{1,3}$ has density at least ${4}/{11}$.
\end{thm}

In the remainder of this section, we provide constructions of identifying codes
in circulant graphs $\cg{n}{1,3}$. From Theorem~\ref{thm:idmain}, we know that
such codes have size at least $\ceil{4n/11}$. We give general constructions for
$n\ge11$. These codes have size $\ceil{4n/11}$, unless when $n\equiv8\mod11$, where
the constructed code has size $\ceil{4n/11}+1$. We do not know whether this is
best possible, but using a brute-force computer search, we verified that
for $n=19,30,41$, an identifying code of size $\ceil{4n/11}$ does not exist.
For {$n<13$}, we verified using this program that
$\id\big(\cg{7}{1,3}\big)=\id\big(\cg{9}{1,3}\big)=\id\big(\cg{10}{1,3}\big)=4$,
and $\id\big(\cg{8}{1,3}\big)=6$.

For a nonnegative integer $t$, let
\[B_t=\big\{11i+j\st 0\le i\le t-1\text{ and }j\in\{0,4,5,6\}\big\}.\]
In particular, $B_0=\emptyset$.
It is easy to see that $B_t$ is indeed an identifying code in $\cg{11t}{1,3}$.
The sets $B_t$ can indeed be used in constructions of identifying codes for
the graphs $\cg{n}{1,3}$ when $n$ is not necessarily a multiple of~$11$.
Such constructions are presented in Table~\ref{tbl:id}.
\begin{table}[ht]
\centering
\begin{tabular}{lcl}
\hline
$n$ &&  An identifying code for $\cg{n}{1,3}$\\
\hline
$11t$    &&  $B_t$\\
$11t+1$  && $B_t\cup\{11t-4\}$\\
$11t+2$  && $B_{t-1}\cup\{11t-11,11t-10,11t-5,11t-4,11t-1\}$\\
$11t+3$  && $B_t\cup\{11t, 11t+1\}$\\
$11t+4$  && $B_t\cup\{11t, 11t+1\}$\\
$11t+5$  && $B_{t-1}\cup\{11t-11,11t-10,11t-5,11t-4,11t+1,11t+2\}$\\
$11t+6$  && $B_t\cup\{11t, 11t+1,11t+4\}$\\
$11t+7$  && $B_t\cup\{11t, 11t+1,11t+4\}$\\
$11t+8$  && $B_t\cup\{11t, 11t+1,11t+6,11t+7\}$\\
$11t+9$  && $B_t\cup\{11t, 11t+1,11t+6,11t+7\}$\\
$11t+10$ && $B_t\cup\{11t, 11t+1,11t+6,11t+7\}$\\
\hline
\end{tabular}
\caption{Constructions of identifying codes for the circulant graphs $\cg{n}{1,3}$.
         Here $t$ is a positive integer.}
\label{tbl:id}
\end{table}

We omit the proofs here. The proof are straight-forward, and all take advantage
of the ``local'' structure of the graphs $\cg{n}{1,3}$, namely the fact that
each neighbourhood is contained in an interval of length~$6$.
We present an example of these codes in Figure~\ref{fig:ideg}.
These results are summarized in the next theorem.

\begin{figure}[ht]
\centering
\includegraphics{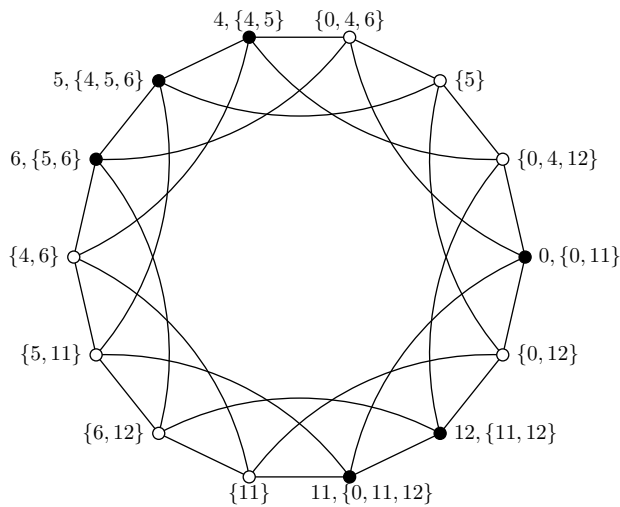}
\caption{A minimum identifying code of $\cg{14}{1,3}$.
        Vertices in the code are in black.
        The number next to a vertex is its label.
        The set next to each vertex is its shadow on this code.}
\label{fig:ideg}
\end{figure}

\begin{thm}
Let {$n\ge9$}. Then $\id\big(\cg{n}{1,3}\big)=\ceil{4n/11}$ if $n\not\equiv8\mod11$,
and ${\ceil{4n/11}\le}\id\big(\cg{n}{1,3}\big)\le\ceil{4n/11}+1$ if $n\equiv8\mod11$.
\label{thm:idconst}
\end{thm}

For an identifying code in $\cg{\infty}{1,3}$ with density $4/11$, one may take
the code \[B_\infty=\big\{11i+j\st i\in\Z\text{ and }j\in\{0,4,5,6\}\big\}.\]

\section{Concluding remarks}

Determining locating and identifying numbers of general circulant graphs
remain open. In particular the circulant graphs $\cg{n}{1,d}$ with $d\ge4$
are of interest. For larger values of $d$, proofs similar to those presented
in this paper get too complicated, so a new approach seems necessary.
We close this article by two problems involving the only graphs $\cg{n}{1,3}$
whose exact locating/identifying number is not settled here.

\begin{pr}
Show that if $n\ge13$ and $n\equiv2\mod6$, the circulant graph $\cg{n}{1,3}$ does not
admit a locating code of size $\ceil{n/3}$.
\end{pr}

\begin{pr}
Show that if $n\equiv8\mod11$, the circulant graph $\cg{n}{1,3}$ does not
admit an identifying code of size $\ceil{4n/11}$.
\end{pr}

\def\soft#1{\leavevmode\setbox0=\hbox{h}\dimen7=\ht0\advance \dimen7
  by-1ex\relax\if t#1\relax\rlap{\raise.6\dimen7
  \hbox{\kern.3ex\char'47}}#1\relax\else\if T#1\relax
  \rlap{\raise.5\dimen7\hbox{\kern1.3ex\char'47}}#1\relax \else\if
  d#1\relax\rlap{\raise.5\dimen7\hbox{\kern.9ex \char'47}}#1\relax\else\if
  D#1\relax\rlap{\raise.5\dimen7 \hbox{\kern1.4ex\char'47}}#1\relax\else\if
  l#1\relax \rlap{\raise.5\dimen7\hbox{\kern.4ex\char'47}}#1\relax \else\if
  L#1\relax\rlap{\raise.5\dimen7\hbox{\kern.7ex
  \char'47}}#1\relax\else\message{accent \string\soft \space #1 not
  defined!}#1\relax\fi\fi\fi\fi\fi\fi}

\end{document}